\documentclass[12pt]{amsart}
\usepackage{amsmath, amsfonts, amssymb, amsthm,hyperref,mathtools,array}
\hypersetup{hypertex=true,
	colorlinks=true,
	linkcolor=blue,
	anchorcolor=blue,
	citecolor=blue}
\usepackage{bm}
\allowdisplaybreaks[4]
\def\({\bg(}
\def\){\bg)}

\def\ord{{\rm ord}}
\def\sgn{{\rm sgn}}

\def\Gal{{\rm Gal}}

\def\v{{\bm v}}

\def\x{{\bm x}}
\def\alg{{\rm alg}}
\def\diag{{\rm diag}}
\def\sgn{{\rm sign}}
\def\Re{{\rm Re}}
\def\B{{\rm B}}
\def\pmod #1{\ ({\rm{mod}}\ #1)}
\def\mod #1{\ {\rm mod}\ #1}
\def\Ack{\medskip\noindent {\bf Acknowledgments}}

\theoremstyle{plain}
\newtheorem{theorem}{Theorem}[section]
\newtheorem{lemma}{Lemma}
\newtheorem{corollary}{Corollary}

\theoremstyle{definition}

\theoremstyle{remark}
\newtheorem{remark}{Remark}

\makeatletter
\@namedef{subjclassname@2020}{%
	\textup{2020} Mathematics Subject Classification}
\makeatother
\vspace{4mm}

\begin{document}
	\medskip
	\hbox{\it accepted for publication in Finite Fields and Their Applications}
	\title[H.-L. Wu and L.-Y. Wang]
	{The Gross-Koblitz formula and almost circulant matrices related to Jacobi sums}
	\author[H.-L. Wu and L.-Y. Wang]{Hai-Liang Wu and Li-Yuan Wang}
	
	\address {(Hai-Liang Wu) School of Science, Nanjing University of Posts and Telecommunications, Nanjing 210023, People's Republic of China}
	\email{\tt whl@njupt.edu.cn}
	
	\address {(Li-Yuan Wang) School of Physical and Mathematical Sciences, Nanjing Tech University, Nanjing 211816, People's Republic of China}
	\email{\tt lywang@njtech.edu.cn}
	
	\keywords{Jacobi sums, Gauss Sums, The Gross-Koblitz Formula, Cyclotomic Matrices.
		\newline \indent This work was supported by the Natural Science Foundation of China (Grant Nos. 12101321 and 12201291).}
	
	\begin{abstract}
		In this paper, we mainly consider arithmetic properties of the cyclotomic matrix $B_p(k)=\left[J_p(\chi^{ki},\chi^{kj})^{-1}\right]_{1\le i,j\le (p-1-k)/k}$, where $p$ is an odd prime, $1\le k<p-1$ is a divisor of $p-1$, $\chi$ is a generator of the group of all multiplicative characters of the finite field $\mathbb{F}_p$ and $J_p(\chi^{ki},\chi^{kj})$ is Jacobi sum over $\mathbb{F}_p$. By using the Gross-Koblitz formula and some $p$-adic tools, we first prove that 
		$$p^{n-2}\det B_p(k)\equiv (-1)^{\frac{(n-1)(p+n-3)}{2}}	\left(\frac{1}{k!}\right)^{n-2}\frac{1}{(2k)!}\pmod {p},$$
		where $p-1=kn$. By establishing some theories on almost circulant matrices, we show that 
		$$\det B_p(k)=(-1)^{\frac{(n-1)(p+n-1)}{2}}p^{-(n-1)}n^{n-2}a_p(k).$$
		 Here $a_p(k)$ is the coefficient of $t$ in the minimal polynomial of $\sum_{y\in U_k}(e^{2\pi{\bf i}y/p}-1)$, where $U_k$ is the set of all $k$-th roots of unity over $\mathbb{F}_p$. Also, for $k=1,2$ we obtain explicit expressions of $\det B_p(k)$. 
	\end{abstract}
	\maketitle
	
	\tableofcontents

	\section{Introduction}
	\setcounter{lemma}{0}
	\setcounter{theorem}{0}
	\setcounter{equation}{0}
	\setcounter{conjecture}{0}
	\setcounter{remark}{0}
	\setcounter{corollary}{0}
	
	\subsection{Notation}
	
	Throughout this paper, let $p$ be an odd prime and let $\mathbb{Q}_p$ be the $p$-adic number field. Let $\mathbb{Z}_p$ be the ring of all $p$-adic integers over $\mathbb{Q}_p$ and let $\mathbb{Z}_p^{\times}$ be the group of all $p$-adic units over $\mathbb{Q}_p$. The symbol  $\mathbb{Q}_p^{\alg}$ denotes an algebraic closure of $\mathbb{Q}_p$ and $\mathbb{C}_p$ denotes the completion of $\mathbb{Q}_p^{\alg}$. Also, let 
	$$\ord_p:\ \mathbb{C}_p\rightarrow \mathbb{R}$$
	be the extension of the $p$-adic order function over $\mathbb{Q}_p$. 
	
	Let $\mathbb{F}_p$ be the finite field with $p$ elements. Let $\mathbb{F}_p^{\times}=\mathbb{F}_p\setminus\{0\}$ be the multiplicative group of all nonzero elements of $\mathbb{F}_p$, and let $\widehat{\mathbb{F}_p^{\times}}$ be the cyclic group of all multiplicative characters of $\mathbb{F}_p$. For any multiplicative character $\psi: \mathbb{F}_p^{\times}\rightarrow\mathbb{C}$ (or $\mathbb{C}_p$), we extend $\psi$ to $\mathbb{F}_p$ by defining $\psi(0)=0$. The trivial character is denoted by $\varepsilon$. 
	
	Let $\zeta_p\in\mathbb{C}$ (or $\mathbb{C}_p$) be a primitive $p$-th root of unity. Then the Gauss sum of $\psi$ is defined by 
	\begin{equation*}
		G_p(\psi):=\sum_{x\in\mathbb{F}_p}\psi(x)\zeta_p^x.
	\end{equation*}
	Also, for any $A,B\in\widehat{\mathbb{F}_p^{\times}}$, the Jacobi sum of $A,B$ is defined by 
	\begin{equation*}
		J_p(A,B)=\sum_{x\in\mathbb{F}_p}A(x)B(1-x).
	\end{equation*}
	
	In addition, given an $n\times n$ matrix $M$ over a commutative ring, we use $\det M$ to denote the determinant of $M$. 
	
	\subsection{Background}
	
	Determining the explicit values of Gauss sums and Jacobi sums is a classical topic in number theory. Readers may refer to \cite{BEK} for the history on this topic. In 1811, Gauss first obtained the explicit values of quadratic Gauss sums over $\mathbb{F}_p$, that is, 
	\begin{equation*}
		\sum_{x\in\mathbb{F}_p}\left(\frac{x}{p}\right)e^{2\pi{\bf i}x/p}=\sqrt{(-1)^{\frac{p-1}{2}}p},
	\end{equation*}
	where $(\frac{\cdot}{p})$ is the Legendre symbol and ${\bf i}\in\mathbb{C}$ is a primitive $4$-th root of unity with argument $\pi/2$. 
	
	In 1978 Loxton \cite{Loxton} posed a conjecture on the exact values of quartic Gauss sums over $\mathbb{F}_p$. In fact, suppose $p\equiv 1\pmod 4$. Then there exist a unique integer $a$ and an integer $b$ such that 
	$$p=a^2+b^2$$ 
	with $a\equiv -1\pmod 4$ and $p-5-2b\equiv 4\pmod 8$. Let $L=\mathbb{Q}(e^{2\pi{\bf i}/p},e^{2\pi{\bf i}/(p-1)})$ and let $\mathcal{O}_L$ be the ring of all algebraic integers over $L$. Let $\mathfrak{p}$ be a prime ideal of $\mathcal{O}_L$ with $p\in\mathfrak{p}$. Then it is easy to see that 
	$$\mathcal{O}_L/\mathfrak{p}\cong\mathbb{F}_p.$$
	Let $\omega_{\mathfrak{p}}\in\widehat{\mathbb{F}_p^{\times}}$ be the Teich\"{u}muller character of $\mathfrak{p}$, i.e., 
	$$\omega_{\mathfrak{p}}(x)\equiv x\pmod {\mathfrak{p}}$$
	for any $x\in\mathcal{O}_L$. Also, it is known that $\omega_{\mathfrak{p}}$ is a generator of $\widehat{\mathbb{F}_p^{\times}}$. Let $\chi_{\mathfrak{p}}=\omega_{\mathfrak{p}}^{-(p-1)/4}$ be a character of order $4$. Then Loxton \cite{Loxton} conjectured that the quartic Gauss sum
	\begin{equation*}
		\sum_{x\in\mathbb{F}_p}\chi_{\mathfrak{p}}(x)e^{2\pi{\bf i}x/p}
		=c_p\left(\frac{|b|}{|a|}\right)(-1)^{(b^2+2b)/8}p^{1/4}(a+b{\bf i})^{1/2},
	\end{equation*}
	where $\Re(a+b{\bf i})^{1/2}>0$, $(\frac{\cdot}{|a|})$ is the Jacobi symbol and $c_p\in\{\pm 1\}$ is determined by the congruence 
	\begin{equation*}
		c_p\equiv \frac{|b|}{a}\left(\frac{p-1}{2}\right)!\left(\frac{2}{p}\right)\pmod p.		
	\end{equation*}
	One year later, Matthews \cite{Matthews} confirmed Loxton’s challenging  conjecture. Hence almost $175$ years after Gauss’s determination of quadratic Gauss sums, an elegant formula for quartic Gauss sum has been found. 
	
	In 1979 Gross and Koblitz \cite{Gross-Koblitz} obtained the well-known Gross-Koblitz formula which reveals the relations between the explict values of Gauss sums and the $p$-adic $\Gamma$ function (we will give a detailed introduction in Section 2). 
	
	We now turn to the cyclotomic matrices. Lehmer \cite{Lehmer} and Carlitz \cite{Carlitz} initiated the study of cyclotomic matrices. For example, Carlitz considered the matrix 
	\begin{equation*}
		C_p=\left[\psi(i+j)\right]_{1\le i,j\le p-1},
	\end{equation*}
	where $\psi\in\widehat{\mathbb{F}_p^{\times}}$ is not a trivial character. By using the theory of circulant matrices and character sums over $\mathbb{F}_p$, Carlitz \cite[Theorem 5]{Carlitz} proved that 
	\begin{equation*}
		\det C_p=\begin{cases}
			\frac{1}{p}(-1)^{\frac{p-1}{2f}}G_p(\psi)^{p-1} & \mbox{if}\ f\equiv 1\pmod 2,\\
			\frac{1}{p}(-1)^{\frac{p-1}{f}}\delta(\psi)^{p-1}G_p(\psi)^{p-1} & \mbox{if}\ f\equiv 0\pmod 2,
		\end{cases}
	\end{equation*}
	where $f=\min\{k\in\mathbb{Z}^+: \psi^k=\varepsilon\}$ is the order of $\psi$ and 
	\begin{equation*}
		\delta(\psi)=\begin{cases}
			1         & \mbox{if}\ \psi(-1)=1,\\
			-{\bf i}  & \mbox{if}\  \psi(-1)=-1.
		\end{cases}
	\end{equation*}
	This indicates a close relationship between this type of matrices and Gauss sums.
	
	Along this line, Chapman \cite{Chapman} and Vsemirnov \cite{Vsemirnov12,Vsemirnov13} further investigated some interesting variants of Carlitz's matrix $C_p$. The most well-known result among these variants is Vsemirnov's ingenious proof of Chapman's ``evil determinant conjecture". This challenging conjecture states that 
	\begin{equation*}
		\det\left[\left(\frac{j-i}{p}\right)\right]_{1\le i,j\le \frac{p+1}{2}}=
		\begin{cases}
			-a_p & \mbox{if}\ p\equiv 1\pmod 4,\\
			1     & \mbox{if}\ p\equiv 3\pmod 4.
		\end{cases}
	\end{equation*}
	The rational number $a_p$ is defined by 
	$$\epsilon_p^{(2-(\frac{2}{p}))h_p}=a_p+b_p\sqrt{p}\ (a_p,b_p\in\mathbb{Q}),$$ 
	where $\epsilon_p>1$ and $h_p$ are the fundamental unit and the class number of $\mathbb{Q}(\sqrt{p})$ respectively. 
	
	In 2019 Sun \cite{Sun19} and Krachun and his collaborators \cite{Krachun} considered the cyclotomic matrices involving quadratic polynomials over $\mathbb{F}_p$. For example, for $c,d\in\mathbb{F}_p$ with $d\neq 0$ and $c^2-4d\neq 0$, Sun \cite[Theorem 1.3]{Sun19} defined the matrix 
	\begin{equation*}
		(c,d)_p=\left[\left(\frac{i^2+cij+dj^2}{p}\right)\right]_{1\le i,j\le p-1},
	\end{equation*}
	and proved that $(c,d)_p$ is singular whenever $(\frac{d}{p})=-1$. For the case $(\frac{d}{p})=1$, it is easy to verify that 
	\begin{equation*}
		\det (c,d)_p=\left(\frac{\sqrt{d}}{p}\right)\cdot \det (g,1)_p,
	\end{equation*}
	where $\sqrt{d}\in\mathbb{F}_p$ with $(\sqrt{d})^2=d$ and $g=c/\sqrt{d}$. For $g\neq \pm 2$, Wu \cite{Wu21} proved that $(g,1)_p$ is a singular matrix whenever the curve defined by the equation 
	\begin{equation*}
		y^2=x(x^2+gx+1)
	\end{equation*}
	is a supersingular elliptic curve over $\mathbb{F}_p$. 
	
	\subsection{Motivation} After introducing the above relevant results, we now describe our research motivations. For any complex numbers $x,y$ with $\Re(x)>0$ and $\Re(y)>0$, the Gamma function is defined by 
	\begin{equation*}
		\Gamma(x)=\int_{0}^{\infty}t^{x-1}e^{-t}dt.
	\end{equation*}
	The Beta function is defined by 
	\begin{equation*}
		\B(x,y)=\int_{0}^{1}t^{x-1}(1-t)^{y-1}dt.
	\end{equation*}
	Then it is well-known that the function 
	\begin{equation*}
		G_p: \widehat{\mathbb{F}_p^{\times}}\rightarrow\mathbb{C}\ (\text{or}\ \mathbb{C}_p)
	\end{equation*}
	by sending $\psi$ to the Gauss sum $G_p(\psi)$ is a finite field analogue of the Gamma function. Also, the function 
	\begin{equation*}
		J_p: \widehat{\mathbb{F}_p^{\times}}\times\widehat{\mathbb{F}_p^{\times}}\rightarrow \mathbb{C}\ (\text{or}\ \mathbb{C}_p)
	\end{equation*}
	by sending $(A,B)$ to the Jacobi sum $J_p(A,B)$ is a finite field analogue of the Beta function. Readers may refer to the survey paper \cite{F} for details on this topic. 
	
	Let $n\in\mathbb{Z}^+$. Then it is known that (see \cite[(4.5) and (4.8)]{Normand})
	\begin{equation}\label{Eq. example 1 for det of Gamma}
		\det\left[\Gamma(i+j)\right]_{1\le i,j\le n}=\prod_{r=0}^{n-1}r!(r+1)!,
	\end{equation}
	and that 
	\begin{equation}\label{Eq. example 2 for det of Gamma}
		\det\left[\frac{1}{\Gamma(i+j)}\right]_{1\le i,j\le n}=(-1)^{\frac{n(n-1)}{2}}\prod_{r=0}^{n-1}\frac{r!}{(n+r)!}.
	\end{equation}
	Letting $\chi$ be a generator of $\widehat{\mathbb{F}_p^{\times}}$, as finite field analogues of (\ref{Eq. example 1 for det of Gamma}) and (\ref{Eq. example 2 for det of Gamma}), Wu, Li, Wang and Yip \cite{WLWY} proved that 
	\begin{equation}\label{Eq. finite field analogue 1 for Gamma}
		\det\left[G_p(\chi^{i+j})\right]_{0\le i,j\le p-2}=(-1)^{\frac{p-3}{2}}(p-1)^{p-1},
	\end{equation}
	and that 
	\begin{equation}\label{Eq. finite field analogue 2 for Gamma}
		\det\left[G_p(\chi^{i+j})^{-1}\right]_{0\le i,j\le p-2}=\frac{(-1)^{\frac{p(p+1)}{2}}(p-1)^{p-1}}{p^{p-2}}.
	\end{equation}
	Also, for the Beta function, it is known that 
	\begin{equation*}
		\B(i,j)=\frac{\Gamma(i)\Gamma(j)}{\Gamma(i+j)}.
	\end{equation*}
	By this and \cite[(4.8)]{Normand} one can verify that 
	\begin{equation}\label{Eq. example 1 for det of Beta}
		\det\left[\B(i,j)\right]_{1\le i,j\le n}=(-1)^{\frac{n(n-1)}{2}}\prod_{r=0}^{n-1}\frac{(r!)^3}{(n+r)!},
	\end{equation}
	and that 
	\begin{equation}\label{Eq. example 2 for det of Beta}
		\det\left[\frac{1}{\B(i,j)}\right]_{1\le i,j\le n}=n!.
	\end{equation}
	For the finite field analogue of (\ref{Eq. example 1 for det of Beta}), Wu, Wang and Pan \cite{WWP} proved that 
	\begin{equation}\label{Eq. finite field analogue 1 for Beta}
		\det\left[J_p(\chi^i,\chi^j)\right]_{1\le i,j\le p-2}=(p-1)^{p-3}.
	\end{equation}
	
	Motivated by the above results, in this paper, we focus on the finite field analogue of (\ref{Eq. example 2 for det of Beta}). In general, let $1\le k<p-1$ be a divisor of $p-1$ and let $\chi$ be a generator of $\widehat{\mathbb{F}_p^{\times}}$, we define the matrix 
	\begin{equation}\label{Eq. definition of Bp(k)}
		B_p(k)=\left[\frac{1}{J_p(\chi^{ki},\chi^{kj})}\right]_{1\le i,j\le (p-1-k)/k},
	\end{equation}
	which involves all nontrivial $k$-th multiplicative characters of $\mathbb{F}_p$. We will see later that $B_p(k)$ is much more complicated than the previous matrices defined in (\ref{Eq. finite field analogue 1 for Gamma}),(\ref{Eq. finite field analogue 2 for Gamma}) and (\ref{Eq. finite field analogue 1 for Beta}). Also, the methods used in this paper are completely different from before. 
	
	\subsection{Main Theorems} We now state our main results of this paper. We first consider some local properties of $\det B_p(k)$.
	
	\begin{theorem}\label{Thm. local properties of Bp(k)}
		Let $p$ be an odd prime and let $\chi$ be a generator of $\widehat{\mathbb{F}_p^{\times}}$. Let $1\le k<p-1$ be a divisor of $p-1$. Then the following results hold.
		
		{\rm (i)} $\det B_p(k)\in\mathbb{Q}$ and is independent of the choice of the generator $\chi$. 
		
		{\rm (ii)} Let $p-1=kn$. Then $\ord_p\left(\det B_p(k)\right)=-(n-2)$ and 
		$$p^{n-2}\det B_p(k)\equiv (-1)^{\frac{(n-1)(p+n-3)}{2}}	\left(\frac{1}{k!}\right)^{n-2}\frac{1}{(2k)!}\pmod {p\mathbb{Z}_p}.$$
	\end{theorem}
	
	We next consider the explicit value of $\det B_p(k)$. Using the above notations, let 
	$$U_k=\{y\in\mathbb{F}_p: y^k=1\}$$
	be the subgroup of all $k$-th roots of unity over $\mathbb{F}_p$, and let 
	$$\theta_k=\left(-k+\sum_{y\in U_k}\zeta_p^y\right)=\sum_{y\in U_k}\left(\zeta_p^y-1\right).$$
	Also, we let $P_k(t)$ be the minimal polynomial of $\theta_k$ over $\mathbb{Q}$. Adopting the above notations, we have the following theorem.
	
	\begin{theorem}\label{Thm. global properties of Bp(k)}
		Let $p$ be an odd prime and let $\chi$ be a generator of $\widehat{\mathbb{F}_p^{\times}}$. Let $1\le k<p-1$ be a divisor of $p-1$ and write $p-1=kn$. Then the following results hold.
		
		{\rm (i)}  Suppose that $a_p(k)$ is the coefficient of $t$ in the polynomial $P_k(t)$. Then 
		$$\det B_p(k)=(-1)^{\frac{(n-1)(p+n-1)}{2}}p^{-(n-1)}n^{n-2}a_p(k).$$
		
		{\rm (ii)} For $k=1$, we have 
		$$\det B_p(1)=\frac{1}{2}(p-1)^{p-2}p^{-(p-3)}.$$
		
		{\rm (iii)} Suppose $p\ge5$. Then 
		$$ \det B_p(2)=\frac{(-1)^{\frac{p+1}{2}+\lfloor\frac{p-3}{4}\rfloor}}{24}\left(\frac{p-1}{2}\right)^{\frac{p-5}{2}}(p^2-1)p^{-(p-5)/2},$$
		where $\lfloor\cdot\rfloor$ is the floor function. 
	\end{theorem}
	
	\begin{remark} (i) Note that 
		$$(-1)^{\frac{p+1}{2}+\lfloor\frac{p-3}{4}\rfloor}=\left(\frac{-2}{p}\right).$$
		
		(ii) In Section 5, we will see that $\theta_k$ is indeed a generator of the unique intermediate field $L_n$ of $\mathbb{Q}(\zeta_p)$ with $[L_n:\mathbb{Q}]=n$, that is, $\mathbb{Q}(\theta_k)=L_n$. 
		
		(iii) Although Theorem \ref{Thm. global properties of Bp(k)}(i) states that  $\det B_p(k)=(-1)^{\frac{(n-1)(p+n-1)}{2}}p^{-(n-1)}n^{n-2}a_p(k)$, for an arbitrary divisor $k$ of $p-1$ with $3\le k<p-1$, finding a simple expression of $\det B_p(k)$ like (ii) and (iii) in Theorem \ref{Thm. global properties of Bp(k)} seems difficult. Even for the case of $k=3$, we cannot get a satisfactory expression. 
		The following table lists the values of $a_p(3)$ for $p=7,13,19,31,37,43,61,67$.
		\begin{table}[h]
			\centering
			\begin{tabular}{|c|c|c|c|c|c|c|c|c|}
				\hline
				$p$          & $7$ & $13$   & $19$    & $31$      & $37$       & $43$       & $61$         & $67$\\ \hline
				$a_p(3)$     & $7$ & $143$  & $1862$  & $196850$  & $1818106$  & $15924276$ & $9251997612$ & $74126282073$\\ \hline
			\end{tabular}
		\end{table}
	\end{remark}
	
	As a direct consequence of Theorems \ref{Thm. local properties of Bp(k)}--\ref{Thm. global properties of Bp(k)}, we have the following result.
	
	\begin{corollary}
		Let notations be as above. Then
		$\ord_p\left(a_p(k)\right)=1$ and 
		$$\frac{a_p(k)}{p}\equiv -\left(\frac{1}{(k-1)!}\right)^{n-2}\frac{1}{(2k)!}\pmod {p\mathbb{Z}_p}.$$
	\end{corollary}
	
	\subsection{Outline of This Paper} In Section 2, we will introduce some preliminaries on Gauss sums, Jacobi sums and the $p$-adic $\Gamma$ functions. The proof of Theorem \ref{Thm. local properties of Bp(k)} will be given in Section 3. In Section 4, we will introduce the definition of almost circulant matrices and pose some necessary lemmas. Finally, we will prove Theorem \ref{Thm. global properties of Bp(k)} in Section 5.
	
	\section{Preliminaries on Gauss sums, Jacobi sums and the $p$-adic $\Gamma$-functions.}
	\setcounter{lemma}{0}
	\setcounter{theorem}{0}
	\setcounter{equation}{0}
	\setcounter{conjecture}{0}
	\setcounter{remark}{0}
	\setcounter{corollary}{0}
	
	We begin with the following basic properties on Gauss sums and Jacobi sums.
	\begin{lemma}\label{Lem. basic properties of Gauss and Jacobi sums}
		Let $\psi,\psi_1,\psi_2\in\widehat{\mathbb{F}_p^{\times}}$ be nontrivial characters. Then the following results hold.
		
		{\rm (i)} $G_p(\psi)G_p(\psi^{-1})=p\psi(-1)$ and $J_p(\psi,\psi^{-1})=-\psi(-1).$
		
		{\rm (ii)} If $\psi_1\psi_2\neq\varepsilon$, then  
		$$J_p(\psi_1,\psi_2)=
		\frac{G_p(\psi_1)G_p(\psi_2)}{G_p(\psi_1\psi_2)}.$$
	\end{lemma}
	
	Recall that $p$ is an odd prime. We next introduce the $p$-adic $\Gamma$ function $\Gamma_p$. For any $n\in\mathbb{Z}^+$, we first define 
	$$\Gamma_p(n)=(-1)^n\prod_{0<k<n, p\nmid k}k\in\mathbb{Z}_p^{\times}.$$
	As $\mathbb{Z}^+$ is dense in $\mathbb{Z}_p$ and $\mathbb{Z}_p^{\times}$ is a closed multiplicative group, the $p$-adic $\Gamma$ function 
	$\Gamma_p: \mathbb{Z}_p \rightarrow \mathbb{Z}_p^{\times}$
	is defined by 
	$$\Gamma_p(x)=\lim_{n\rightarrow x}\Gamma_p(n)$$
	for any $x\in\mathbb{Z}_p$. 
	
	We next list some basic properties of $p$-adic $\Gamma$ function. 
	
	\begin{lemma}\label{Lem. basic properties of p adic gamma}
		Suppose that $p\ge5$ is a prime. Then the following results hold.
		
		{\rm (i)} For any $x\in\mathbb{Z}_p$, we have 
		\begin{equation*}
			\Gamma_p(x+1)=\begin{cases}
				-x\Gamma_p(x)  & \mbox{if}\ x\in\mathbb{Z}_p^{\times},\\
				-\Gamma_p(x)   & \mbox{if}\ x\in p\mathbb{Z}_p.
			\end{cases}
		\end{equation*}
		
		{\rm (ii)} Let $n\in\mathbb{Z}^+$. Then for any $x,y\in\mathbb{Z}_p$ we have 
		$$x\equiv y\pmod{p^n\mathbb{Z}_p}\Rightarrow \Gamma_p(x) \equiv \Gamma_p(y) \pmod{p^n\mathbb{Z}_p}.$$
		
		{\rm (iii)} For any $x\in\mathbb{Z}_p$, write $x=x_0+px_1$ with $x_0\in\{1,2,\cdots,p\}$ and $x_1\in\mathbb{Z}_p$. Then 
		$$\Gamma_p(x)\Gamma_p(1-x)=(-1)^{x_0}.$$
	\end{lemma}
	
	We now state a relationship between the $p$-adic $\Gamma$ function and Gauss sums. Let $\pi\in\mathbb{C}_p$ be a root of the equation $x^{p-1}+p=0$. Let $\zeta_{\pi}\in\mathbb{C}_p$ be the unique primitive $p$-th root of unity such that $\zeta_{\pi}\equiv 1+\pi\pmod{\pi^2}$. Also, since 
	$$\mathbb{Z}_p/p\mathbb{Z}_p\cong\mathbb{Z}/p\mathbb{Z}=\mathbb{F}_p,$$
	the Teich\"{u}muller character $\omega_p: \mathbb{F}_p\rightarrow \mathbb{Z}_p$ is defined by 
	$$\omega_p(x\mod p\mathbb{Z}_p)\equiv x\pmod {p\mathbb{Z}_p}$$
	for any $x\in\mathbb{Z}_p$. It is easy to verify that $\omega_p$ is a generator of $\widehat{\mathbb{F}_p^{\times}}$. The following Gross-Koblitz formula \cite[Theorem 1.7]{Gross-Koblitz} gives the explicit value of the Gauss sum involving $\omega_p$. 
	
	\begin{lemma}[The Gross-Koblitz Formula]\label{Lem. Gross-Koblitz}
		Let notations be as above. Then for $0\le j\le p-2$,
		$$G_p(\omega_p^{-j})=\sum_{x\in\mathbb{F}_p}\omega_p^{-j}(x)\zeta_{\pi}^x=-\pi^j\cdot\Gamma_p\left(\frac{j}{p-1}\right).$$
	\end{lemma}

	\section{Proof of Theorem \ref{Thm. local properties of Bp(k)}}
	\setcounter{lemma}{0}
	\setcounter{theorem}{0}
	\setcounter{equation}{0}
	\setcounter{conjecture}{0}
	\setcounter{remark}{0}
	\setcounter{corollary}{0}
	
	Recall that $p$ is an odd prime and that $\chi$ is a generator of $\widehat{\mathbb{F}_p^{\times}}$. For any divisor $k$ of $p-1$ with $1\le k<p-1$, we set $p-1=kn$. Consider the set  
	$$H_k=\{\chi^{kj}:\ 1\le j\le n-1\}$$
	of all nontrivial $k$-th characters. Let $s\in\mathbb{Z}$ with $\gcd(s,p-1)=1$. Then the map $\chi^{kj}\mapsto\chi^{ksj}$ induces a permutation $\sigma_{s}$ of $H_k$. If we use the symbol $\sgn(\sigma_{s})$ to denote the sign of the permutation $\sigma_{s}$, then the following lemma gives the explicit value of $\sgn(\sigma_{s})$.
	
	\begin{lemma}\label{Lem. permutation}
		Let notations be as above. Then 
		\begin{equation*}
			\sgn(\sigma_{s})=
			\begin{cases}
				(\frac{s}{n})  & \mbox{if}\ n\equiv 1\pmod 2,\\
				1              & \mbox{if}\ n\equiv 2\pmod 4,\\
				(-1)^{(s-1)/2} & \mbox{if}\ n\equiv 0\pmod 4,
			\end{cases}
		\end{equation*}
		where $(\frac{\cdot}{n})$ is the Jacobi symbol if $n$ is odd. In particular, we have 
		$$\sgn(\sigma_{-1})=(-1)^{\frac{(n-1)(n-2)}{2}}.$$
	\end{lemma}
	
	\begin{proof}
		Observe that the permutation $\sigma_{s}$ indeed induces a permutation $\tau_{s}$ of $\mathbb{Z}/n\mathbb{Z}\setminus\{0\mod n\mathbb{Z}\}$, which sends $x\mod n\mathbb{Z}$ to $sx\mod n\mathbb{Z}$. Lerch \cite{Lerch} determined the sign of this permutation, that is, 
		$$\sgn(\tau_{s})=\begin{cases}
			(\frac{s}{n})  & \mbox{if}\ n\equiv 1\pmod 2,\\
			1              & \mbox{if}\ n\equiv 2\pmod 4,\\
			(-1)^{(s-1)/2} & \mbox{if}\ n\equiv 0\pmod 4.
		\end{cases}$$
		Now our desired result follows directly from  $\sgn(\sigma_{s})=\sgn(\tau_{s})$. This completes the proof.		
	\end{proof}
	
	We also need the following lemma, which we will frequently use in the subsequent proofs.
	
	\begin{lemma}\label{Lem. transformation lemma}
		Let notations be as above. Then 
		\begin{equation*}
			\det B_p(k)=(-1)^{\frac{(n-1)(p+n-3)}{2}}p^{-(n-1)}\det D_p(k),
		\end{equation*}
		where $D_p(k)=[a_{ij}]_{1\le i,j\le n-1}$ with 
		$$a_{ij}=\begin{cases}
			-p                & \mbox{if}\ 1\le i=j\le n-1,\\
			G_p(\chi^{ki-kj}) & \mbox{if}\ 1\le i\neq j\le n-1.  
		\end{cases}$$
	\end{lemma}
	
	\begin{proof}
		By Lemma \ref{Lem. basic properties of Gauss and Jacobi sums} we first have 
		\begin{equation*}
			\frac{1}{J_p(\chi^{ki},\chi^{-kj})}=
			\begin{cases}
				\frac{-p}{G_p(\chi^{ki})G_p(\chi^{-kj})}  & \mbox{if}\ 1\le i=j\le n-1,\\ \\
				\frac{G_p(\chi^{ki-kj})}{G_p(\chi^{ki})G_p(\chi^{-kj})}  & \mbox{if}\ 1\le i\neq j\le n-1.
			\end{cases}
		\end{equation*}
		By this and Lemma \ref{Lem. permutation} one can verify that 
		\begin{align*}
			\det B_p(k)
			=&\sgn(\sigma_{-1})\det\left[\frac{1}{J_p(\chi^{ki},\chi^{-kj})}\right]_{1\le i,j\le n-1}\\
			=&(-1)^{\frac{(n-1)(n-2)}{2}}\det\left[\frac{1}{J_p(\chi^{ki},\chi^{-kj})}\right]_{1\le i,j\le n-1}\\
			=&(-1)^{\frac{(n-1)(n-2)}{2}}\prod_{1\le i\le n-1}\frac{1}{G_p(\chi^{ki})G_p(\chi^{-ki})}\cdot\det D_p(k)\\
			=&(-1)^{\frac{(n-1)(n-2)}{2}}(-1)^{\frac{kn(n-1)}{2}}p^{-(n-1)}\det D_p(k)\\
			=&(-1)^{\frac{(n-1)(p+n-3)}{2}}p^{-(n-1)}\det D_p(k).
		\end{align*}
		This completes the proof.
	\end{proof}
	
	Now we are in a position to prove our first theorem.
	
	{\noindent\bf Proof of Theorem \ref{Thm. local properties of Bp(k)}.} (i) Let $\zeta_{p-1}$ be a primitive $(p-1)$-th root of unity. Clearly $\det B_p(k)\in\mathbb{Q}(\zeta_{p-1})$. It is known that the Galois group
	$$\Gal\left(\mathbb{Q}(\zeta_{p-1})/\mathbb{Q}\right)\cong\left(\mathbb{Z}/(p-1)\mathbb{Z}\right)^{\times}.$$
	Now for any $\rho_{s}\in\Gal\left(\mathbb{Q}(\zeta_{p-1})/\mathbb{Q}\right)$ with $\gcd(s,p-1)=1$ and $\rho_{s}(\zeta_{p-1})=\zeta_{p-1}^s$, by Lemma \ref{Lem. permutation} we have 
	\begin{align*}
		\rho_{s}\left(\det B_p(k)\right)
		=&\det\left[\frac{1}{J_p(\chi^{ksi},\chi^{ksj})}\right]_{1\le i,j\le n-1}\\
		=&\sgn(\sigma_{s})^2\cdot\det B_p(k)\\
		=&\det B_p(k).
	\end{align*}
	By the Galois theory, we see that $\det B_p(k)\in\mathbb{Q}$. On the other hand, it is known that every generator $\chi'$ of $\widehat{\mathbb{F}_p^{\times}}$ can be written as $\chi^s$ for some integer $s$ with $\gcd(s,p-1)=1$. Thus, 
	\begin{align*}
		\det \left[\frac{1}{J_p(\chi'^{ki},\chi'^{kj})}\right]_{1\le i,j\le n-1}
		=&\det\left[\frac{1}{J_p(\chi^{ksi},\chi^{ksj})}\right]_{1\le i,j\le n-1}\\
		=&\rho_{s}\left(\det B_p(k)\right)\\
		=&\det B_p(k).
	\end{align*}
	This implies that $\det B_p(k)$ is independent of the choice of the generator $\chi$. This completes the proof of (i).
	
	(ii) The result holds trivially for $p=3$. Now suppose $p\ge 5$.  As proved in (i), we know that $\det B_p(k)$ is independent of the choice of $\chi$. Thus, in the remaining part of this proof, we let $\chi=\omega_p$ be the Teich\"{u}muller character.
	
	By Lemma \ref{Lem. transformation lemma} we first have 
	\begin{equation}\label{Eq. a in the proof of Thm.1}
		\det B_p(k)=(-1)^{\frac{(n-1)(p+n-3)}{2}}p^{-(n-1)}\det D_p(k).
	\end{equation}
	Adopting the notations in Section 2, by Lemma \ref{Lem. Gross-Koblitz} we see that 
	\begin{equation*}
		\det D_p(k)=\det\left[\begin{array}{ccccc}
			\pi^{p-1} & -\pi^k\Gamma_p(1/n) & \cdots & -\pi^{(n-2)k}\Gamma_p((n-2)/n)\\
			-\pi^{(n-1)k}\Gamma_p((n-1)/n) & \pi^{p-1} & \cdots & -\pi^{(n-3)k}\Gamma_p((n-3)/n)\\
			\vdots & \vdots& \ddots & \vdots\\
			-\pi^{2k}\Gamma_p(2/n) & -\pi^{3k}\Gamma_p(3/n) & \cdots & \pi^{p-1}
		\end{array}\right].
	\end{equation*}
	For each $1\le i\le n-2$, extract the factor $\pi^k$ from the $i$-th row. Also, extract the factor $\pi^{2k}$ from the $(n-1)$-th row. Then one can verify that 
	\begin{equation}\label{Eq. b in the proof of Thm.1}
		\det D_p(k)=\pi^{k(n-2)+2k}\det E_p(k)=-p\det E_p(k),
	\end{equation}
	where 
	$$E_p(k)=\left[\begin{array}{ccccc}
		\pi^{(n-1)k} & -\Gamma_p(1/n) & \cdots & -\pi^{(n-3)k}\Gamma_p((n-2)/n)\\
		-\pi^{(n-2)k}\Gamma_p((n-1)/n) & \pi^{(n-1)k} & \cdots & -\pi^{(n-4)k}\Gamma_p((n-3)/n)\\
		\vdots & \vdots& \ddots & \vdots\\
		-\Gamma_p(2/n) & -\pi^{k}\Gamma_p(3/n) & \cdots & \pi^{(n-2)k}
	\end{array}\right].$$
	We next evaluate $\det E_p(k)\mod \pi$. It is clear that 
	\begin{align}\label{Eq. c in the proof of Thm.1}
		\det E_p(k)
		\equiv&\det\left[\begin{array}{ccccc}
			0              & -\Gamma_p(1/n) &  0              & \cdots & 0\\
			0              &  0             & -\Gamma_p(1/n)  & \cdots & 0\\
			\vdots         & \vdots         & \vdots          & \ddots & \vdots\\
			0              & 0              & 0               & \cdots & -\Gamma_p(1/n)\\
			-\Gamma_p(2/n) & 0              & 0               & \cdots & 0
		\end{array}\right]\notag\\
		\equiv& -\Gamma_p\left(\frac{1}{n}\right)^{n-2}\Gamma_p\left(\frac{2}{n}\right)\pmod{\pi}.
	\end{align}
	This implies that $\ord_p(\det E_p(k))=0$. Hence by (\ref{Eq. a in the proof of Thm.1}) and (\ref{Eq. b in the proof of Thm.1}) we first obtain 
	$$\ord_p\left(\det B_p(k)\right)=-(n-2).$$
	
	We now turn to $p^{n-2}\det B_p(k) \mod p$. By (\ref{Eq. a in the proof of Thm.1})--(\ref{Eq. c in the proof of Thm.1}) again we see that 
	\begin{equation}\label{Eq. d in the proof of Thm.1}
		p^{n-2}\det B_p(k)=(-1)^{1+\frac{(n-1)(p+n-3)}{2}}\cdot\det E_p(k). 
	\end{equation}
	Since 
	$$\frac{1}{n}\equiv -k\pmod p\ \text{and}\ \frac{2}{n}\equiv -2k\pmod p,$$
	by Lemma \ref{Lem. basic properties of p adic gamma} and (\ref{Eq. c in the proof of Thm.1}) 
	\begin{align}\label{Eq. e in the proof of Thm.1}
		\det E_p(k)
		\equiv& -\Gamma_p(-k)^{n-2}\cdot\Gamma_p(-2k)\notag\\
		\equiv& \left(\frac{(-1)^{k+1}}{\Gamma_p(k+1)}\right)^{n-2}\frac{1}{\Gamma_p(2k+1)}\notag\\
		\equiv&-\left(\frac{1}{k!}\right)^{n-2}\frac{1}{(2k)!}\pmod {\pi}.
	\end{align}
	
	Noting that $\det E_p(k)\in\mathbb{Z}_p$ and Combining (\ref{Eq. e in the proof of Thm.1}) with (\ref{Eq. d in the proof of Thm.1}), we finally obtain 
	$$p^{n-2}\det B_p(k)\equiv (-1)^{\frac{(n-1)(p+n-3)}{2}}	\left(\frac{1}{k!}\right)^{n-2}\frac{1}{(2k)!}\pmod {p}.$$
	
	In view of the above, we have completed the proof of Theorem \ref{Thm. local properties of Bp(k)}.\qed

	\section{Preliminaries on Almost Circulant Matrices}
	
	\subsection{Circulant Matrices}
	\setcounter{lemma}{0}
	\setcounter{theorem}{0}
	\setcounter{equation}{0}
	\setcounter{conjecture}{0}
	\setcounter{remark}{0}
	\setcounter{corollary}{0}

	We first introduce some basic facts on circulant matrices. Readers may refer to the survey paper \cite{Kra} for more details on circulant matrices. Let $n\ge2$ be an integer and let $\v=(v_0,v_1,\cdots,v_{n-1})\in\mathbb{C}^n$. The circulant matrix $V_n$ of $\v$ is an $n\times n$ matrix defined by 
	$$V_n=[v_{j-i}]_{0\le i,j\le n-1},$$
	where $v_s=v_t$ whenever $s\equiv t\pmod n$. Specifically, 
	$$V_n=\left[\begin{array}{ccccc}
		v_0       &  v_1     & \cdots  & v_{n-2}  & v_{n-1} \\
		v_{n-1}   &  v_0     & \cdots  & v_{n-3}  & v_{n-2} \\
		\vdots    &  \vdots  & \ddots  & \vdots   & \vdots  \\
		v_2       &  v_3     & \cdots  & v_0      & v_1     \\
		v_1       &  v_2     & \cdots  & v_{n-1}  & v_0
	\end{array}\right].$$
	Let $\xi=e^{2\pi{\bf i}/n}$. For $l\in\{0,1,\cdots,n-1\}$, let 
	\begin{equation*}
		\lambda_l=v_0\cdot1+v_1\cdot\xi^l+v_2\cdot\xi^{2l}+\cdots+v_{n-1}\cdot\xi^{(n-1)l},
	\end{equation*}
	and let the column vector 
	\begin{equation*}
		\x_l=\frac{1}{\sqrt{n}}\left(1,\xi^l,\cdots,\xi^{(n-1)l}\right)^T.
	\end{equation*}
	Then it is know that $V_n\x_l=\lambda_l\x_l$ and that $\lambda_0,\cdots,\lambda_{n-1}$ are exactly all the eigenvalues of $V_n$. In addition, we let the matrix 
	\begin{equation}\label{Eq. definition of E}
		E_n=\left(\x_0,\x_1,\cdots,\x_{n-1}\right)=\frac{1}{\sqrt{n}}\left[\xi^{ij}\right]_{0\le i,j\le n-1}=\frac{1}{\sqrt{n}}
		\left[\begin{array}{cccc}
			1      &  1             & \cdots  &   1                \\
			1      & \xi         & \cdots  & \xi^{n-1}        \\
			\vdots & \vdots         & \ddots  & \vdots               \\
			1      & \xi^{n-2}   & \cdots  & \xi^{(n-2)(n-1)}  \\
			1      & \xi^{n-1}   & \cdots  & \xi^{(n-1)(n-1)}  \\
		\end{array}\right].
	\end{equation}
	Then it is easy to verify that $E_n$ is a symmetric unitary matrix, that is, 
	\begin{equation}\label{Eq. E is a symmetric unitary matrix}
		E_n^{-1}=\overline{E_n}.
	\end{equation}
	We also have 
	\begin{equation}\label{Eq. product of Vn and En}
		V_nE_n=E_n\diag\left(\lambda_0,\lambda_1,\cdots,\lambda_{n-1}\right).
	\end{equation}

	\subsection{An Auxiliary Theorem on Almost Circulant Matrices} Now we pose the definition of almost circulant matrices. Let notations be as above. The almost circulant matrix of $\v=(v_0,v_1,\cdots,v_{n-1})$ is an $(n-1)\times (n-1)$ matrix defined by 
	\begin{equation}\label{Eq. definition of almost circulant matrices}
		W_n=\left[v_{j-i}\right]_{1\le i,j\le n-1},
	\end{equation}
	where $v_s=v_t$ whenever $s\equiv t\pmod n$. Then following theorem gives the explicit value of $\det W_n$. 
	
	\begin{theorem}\label{Thm. det of an almost circulant matrix}
		Let notations be as above. Then 
		$$\det W_n=\frac{1}{n}\sum_{l=0}^{n-1}\prod_{k\neq l}\lambda_k.$$
	\end{theorem}
	
	{\noindent\bf Proof of Theorem \ref{Thm. det of an almost circulant matrix}.} We divide our proof into two cases.
	
	{\bf Case I.} $\lambda_0\lambda_1\cdots\lambda_{n-1}\neq 0$.
	
	In this case, let $\widetilde{W_n}=[b_{ij}]_{0\le i,j\le n-1}$ be an $n\times n$ matrix defined by
	$$b_{ij}=\begin{cases}
		1          & \mbox{if}\ i=j=0,\\
		0          & \mbox{if}\ i=0\ \text{and}\ 1\le j\le n-1,\\
		v_{j-i}    &   \mbox{otherwise},
	\end{cases}$$
	i.e., 
	$$\widetilde{W_n}=\left[\begin{array}{ccccc}
		1         &  0       & \cdots  &  0       &  0\\
		v_{n-1}   &  v_0     & \cdots  & v_{n-3}  & v_{n-2} \\
		\vdots    &  \vdots  & \ddots  & \vdots   & \vdots  \\
		v_2       &  v_3     & \cdots  & v_0      & v_1     \\
		v_1       &  v_2     & \cdots  & v_{n-1}  & v_0
	\end{array}\right].$$
	
	For $1\le i\le n-1$ and $0\le l\le n-1$, one can verify that 
	\begin{equation*}
		\frac{1}{\sqrt{n}}\sum_{0\le j\le n-1}v_{j-i}\cdot\xi^{jl}
		=\frac{1}{\sqrt{n}}\sum_{0\le j\le n-1}v_{j}\cdot\xi^{(j+i)l}=\lambda_l\cdot\frac{1}{\sqrt{n}}\xi^{il}.
	\end{equation*}
	This implies that for any $0\le l\le n-1$ we have 
	\begin{align}\label{Eq. 1 in the proof of almost circulant matrix}
		\widetilde{W_n}\x_l
		=&\frac{1}{\sqrt{n}}\left(1,\lambda_l\xi^l,\cdots,\lambda_l\xi^{(n-1)l}\right)^T\notag\\
		=&\lambda_l\left(\x_l+\widetilde{\x_l}\right),
	\end{align}
	where 
	$$\widetilde{\x_l}=\frac{1}{\sqrt{n}}\left(-1+\frac{1}{\lambda_l},0,\cdots,0\right)^T.$$
	By (\ref{Eq. 1 in the proof of almost circulant matrix}) we obtain 
	\begin{equation*}
		\widetilde{W_n}E_n=\left(E_n+\widetilde{E_n}\right)\diag\left(\lambda_0,\lambda_1,\cdots,\lambda_{n-1}\right),
	\end{equation*}
	where $E_n$ is defined by (\ref{Eq. definition of E}) and 
	$$\widetilde{E_n}=\frac{1}{\sqrt{n}}\left[\begin{array}{cccc}
		-1+\frac{1}{\lambda_0}  &  -1+\frac{1}{\lambda_1}  & \cdots & -1+\frac{1}{\lambda_{n-1}}\\
		0                       &   0                      & \cdots & 0\\
		\vdots                  & \vdots                   & \ddots & \vdots\\
		0                       & 0                        & \cdots & 0
	\end{array}\right].$$
	By (\ref{Eq. E is a symmetric unitary matrix}), (\ref{Eq. product of Vn and En}) and the above, we see that 
	$$E_n^{-1}\widetilde{W_n}E_n=\left(I_n+\overline{E_n}\cdot\widetilde{E_n}\right)\diag\left(\lambda_0,\lambda_1,\cdots,\lambda_{n-1}\right).$$
	Hence 
	\begin{equation}\label{Eq. 2 in the proof of almost circulant matrix}
		\det W_n=\det\widetilde{W_n}=\prod_{l=0}^{n-1}\lambda_l\cdot\det \left(I_n+\overline{E_n}\cdot\widetilde{E_n}\right).
	\end{equation}
	Note that 
	$$I_n+\overline{E_n}\cdot\widetilde{E_n}=\frac{1}{n}
	\left[\begin{array}{cccc}
		n-(1-\frac{1}{\lambda_0})  &  -(1-\frac{1}{\lambda_1}) & \cdots & -(1-\frac{1}{\lambda_{n-1}})\\
		-(1-\frac{1}{\lambda_0})   & n-(1-\frac{1}{\lambda_1}) & \cdots & -(1-\frac{1}{\lambda_{n-1}})\\
		\vdots                     & \vdots                    & \ddots & \vdots\\
		-(1-\frac{1}{\lambda_0})   & -(1-\frac{1}{\lambda_1})  & \cdots & n-(1-\frac{1}{\lambda_{n-1}})
	\end{array}\right].$$
	Letting 
	$$\mu=\sum_{l=0}^{n-1}\frac{1}{\lambda_l},$$
	one can verify that 
	\begin{align*}
		\det \left(I_n+\overline{E_n}\cdot\widetilde{E_n}\right)
		&=\frac{1}{n^n}\det \left[\begin{array}{cccc}
			\mu   &  -(1-\frac{1}{\lambda_1}) & \cdots & -(1-\frac{1}{\lambda_{n-1}})\\
			\mu   & n-(1-\frac{1}{\lambda_1}) & \cdots & -(1-\frac{1}{\lambda_{n-1}})\\
			\vdots                     & \vdots                    & \ddots & \vdots\\
			\mu   & -(1-\frac{1}{\lambda_1})  & \cdots & n-(1-\frac{1}{\lambda_{n-1}})
		\end{array}\right]\\
		&=\frac{\mu}{n^n}\det \left[\begin{array}{cccc}
			1       &  0     & \cdots & 0\\
			1       & n      & \cdots & 0\\
			\vdots  & \vdots & \ddots & \vdots\\
			1       & 0      & \cdots & n
		\end{array}\right]\\
		&=\frac{1}{n}\sum_{l=0}^{n-1}\frac{1}{\lambda_l}.
	\end{align*}
	This, together with (\ref{Eq. 2 in the proof of almost circulant matrix}), implies that 
	$$\det W_n=\frac{1}{n}\sum_{l=0}^{n-1}\prod_{k\neq l}\lambda_k.$$
	
	{\bf Case II.} $\lambda_0\lambda_1\cdots\lambda_{n-1}=0$.
	
	In this case, there exists a real number $\gamma>0$ such that $\lambda_l+\delta\neq 0$
	for any $0\le l\le n-1$ and any $0<\delta<\gamma$. For any $0<\delta<\gamma$, let 
	$$\v(\delta)=(v_0+\delta,v_1,\cdots,v_{n-1}),$$
	and let 
	$$V_n(\delta)=\delta I_n+V_n$$
	be the circulant matrix of $\v(\delta)$. Then 
	$$W_n(\delta)=\delta I_{n-1}+W_n$$
	is the almost circulant matrix of $\v(\delta)$. As all eigenvalues of $V_n(\delta)$ are nonzero, by Case I we have 
	$$\det W_n(\delta)=\frac{1}{n}\sum_{l=0}^{n-1}\prod_{k\neq l}\left(\lambda_k+\delta\right).$$
	Then 
	$$\det W_n=\lim_{\delta\rightarrow 0^+}\det W_n(\delta)=\frac{1}{n}\sum_{l=0}^{n-1}\prod_{k\neq l}\lambda_k.$$
	
	In view of the above, we have completed the proof of Theorem \ref{Thm. det of an almost circulant matrix}.\qed

	\section{Proof of Theorem \ref{Thm. global properties of Bp(k)}}
	\setcounter{lemma}{0}
	\setcounter{theorem}{0}
	\setcounter{equation}{0}
	\setcounter{conjecture}{0}
	\setcounter{remark}{0}
	\setcounter{corollary}{0}

	Recall that $p$ is an odd prime and that $\chi$ is a generator of $\widehat{\mathbb{F}_p^{\times}}$. Throughout this section, we let $\zeta_p=e^{2\pi{\bf i}/p}$. For any divisor $k$ of $p-1$ with $1\le k<p-1$, we set $p-1=kn$ and let 
	$$U_k=\{y\in\mathbb{F}_p: y^k=1\}$$
	be the subgroup of all $k$-th roots of unity over $\mathbb{F}_p$. We begin with the following known lemma (see \cite[Lemma 3.5]{WLWY}).
	
	\begin{lemma}\label{Lem. eigenvalues of gauss sums matrices}
		Let notations be as above. For any element $b$ in the quotient group $\mathbb{F}_q^{\times}/U_k$, let 
		\begin{equation*}
			\lambda_b:=n\sum_{y\in U_k}\zeta_p^{by}.
		\end{equation*}
		Then these $\lambda_b$ (where $b\in\mathbb{F}_q^{\times}/U_k$) are exactly all the eigenvalues of $[G_p(\chi^{ki-kj})]_{0\le i,j\le n-1}$. 
	\end{lemma}
	
	Now we are in a position to prove our second theorem. 
	
	{\noindent\bf Proof of Theorem \ref{Thm. global properties of Bp(k)}.} (i) Recall that 
	$$\theta_k=-k+\sum_{y\in U_k}\zeta_p^y=\sum_{y\in U_k}\left(\zeta^y_p-1\right).$$
	We first claim that $\mathbb{Q}(\theta_k)$ is the unique intermediate field of $\mathbb{Q}(\zeta_p)$ with $\left[\mathbb{Q}(\theta_k):\mathbb{Q}\right]=n$. In fact, letting the subgroup 
	$$G_{k}=\left\{\rho\in\Gal\left(\mathbb{Q}(\zeta_p)/\mathbb{Q}\right): \rho(\theta_k)=\theta_k\right\},$$
	it is sufficient to prove that $[\Gal\left(\mathbb{Q}(\zeta_p)/\mathbb{Q}\right):G_k]=n$. 
	Let $\rho_{s}\in\Gal\left(\mathbb{Q}(\zeta_p)/\mathbb{Q}\right)$ with $s\in\mathbb{F}_p^{\times}$ and $\rho_{s}(\zeta_p)=\zeta_p^s$ be an arbitrary element. Noting that $\zeta_p,\zeta_p^2,\cdots,\zeta_p^{p-1}$ are linearly independent over $\mathbb{C}$, one can verify that 
	\begin{align*}
		\rho_{s}\in G_k & \Leftrightarrow \rho_{s}(\theta_k)=\theta_k\\
		& \Leftrightarrow \rho_{s}(\theta_k+k)=\theta_k+k\\
		& \Leftrightarrow \sum_{y\in U_k}\zeta_p^{y}=\sum_{y\in U_k}\zeta_p^{sy}\\
		& \Leftrightarrow s\cdot U_k=\{sy:\ y\in U_k\}=U_k\\
		& \Leftrightarrow s\in U_k.
	\end{align*}
	This implies that $G_k\cong U_k$ and hence $[\Gal\left(\mathbb{Q}(\zeta_p)/\mathbb{Q}\right):G_k]=[\mathbb{F}_p^{\times}:U_k]=n$. Thus, our above claim holds. Recall that $P_k(t)$ is the minimal polynomial of $\theta_k$ over $\mathbb{Q}$. By the above, we obtain 
	\begin{equation}\label{Eq. Pk(t) in the proof of Thm. 2}
		P_k(t)=\prod_{\rho\in\Gal\left(\mathbb{Q}(\theta_k)/\mathbb{Q}\right)}\left(t-\rho(\theta_k)\right)=\prod_{b\in\mathbb{F}_p^{\times}/U_k}\left(t-\theta_k^{(b)}\right),
	\end{equation}
	where 
	$$\theta_k^{(b)}=\rho_{b}(\theta_k)=-k+\sum_{y\in U_k}\zeta_p^{by}$$
	for any element $b\in\mathbb{F}_p^{\times}/U_k$. 
	
	Now let $\widetilde{D_p}(k)=[\widetilde{a_{ij}}]_{0\le i,j\le n-1}$ be an $n\times n$ matrix with 
	$$\widetilde{a_{ij}}=\begin{cases}
		-p                & \mbox{if}\ 0\le i=j\le n-1,\\
		G_p(\chi^{ki-kj}) & \mbox{if}\ 0\le i\neq j\le n-1. 
	\end{cases} $$
	Clearly $\widetilde{D_p}(k)$ is the transpose of the circulant matrix of $\v=(v_0,v_1,\cdots,v_{n-1})$ where 
	$v_0=-p$ and $v_i=G_p(\chi^{ki})$ for $i=1,2,\cdots,n-1$. Also, we have 
	$$\widetilde{D_p}(k)=(-p+1)I_n+\left[G_p(\chi^{ki-kj})\right]_{0\le i,j\le n-1}.$$
	By Lemma \ref{Lem. eigenvalues of gauss sums matrices} we see that 
	$(-p+1)+n\sum_{y\in U_k}\zeta_p^{by}=n\theta_k^{(b)}$ (where $b\in\mathbb{F}_p^{\times}/U_k$) are exactly all the eigenvalues of $\widetilde{D_p}(k)$. Recall that $D_p(k)$ is defined by Lemma \ref{Lem. transformation lemma}. Then it is easy to see that $D_p(k)$ is an almost circulant matrix of $\v$. Now by Theorem \ref{Thm. det of an almost circulant matrix} and (\ref{Eq. Pk(t) in the proof of Thm. 2}) one can verify that 
	$$\det D_p(k)=\frac{1}{n}\sum_{b\in\mathbb{F}_p^{\times}/U_k}\prod_{c\in\mathbb{F}_p^{\times}/U_k\setminus\{b\}}n\theta_k^{(c)}=(-1)^{n-1}n^{n-2}a_p(k).$$
	This, together with Lemma \ref{Lem. transformation lemma}, implies that 
	\begin{equation}\label{Eq. result of (i) of Thm.2}
		\det B_p(k)=(-1)^{\frac{(n-1)(p+n-3)}{2}}p^{-(n-1)}\det D_p(k)=(-1)^{\frac{(n-1)(p+n-1)}{2}}p^{-(n-1)}n^{n-2}a_p(k).
	\end{equation}
	This completes the proof of (i).
	
	(ii) For $k=1$, we have $\theta_1=\zeta_p-1$. Let $\Phi_p(t)$ be the $p$-th cyclotomic polynomial. Then the minimal polynomial $P_1(t)$ of $\theta_1$ is equal to 
	$$\Phi_p(t+1)=1+(1+t)+(1+t)^2+\cdots+(1+t)^{p-1},$$
	and hence 
	$$a_p(1)=\frac{p(p-1)}{2}.$$
	Now by (\ref{Eq. result of (i) of Thm.2}) we obtain 
	$$\det B_p(1)=\frac{1}{2}(p-1)^{p-2}p^{-(p-3)}.$$
	This completes the proof of (ii).
	
	(iii) For $k=2$, we have $\theta_2=\zeta_p+\zeta_p^{-1}-2$. By \cite{BA} we know that the minimal polynomial of $\zeta_p+\zeta_p^{-1}$ is equal to 
	$$\Psi_p(t)=\sum_{r=0}^{(p-1)/2}\binom{p}{2r+1}\left(\frac{t}{4}+\frac{1}{2}\right)^{\frac{p-1}{2}-r}\left(\frac{t}{4}-\frac{1}{2}\right)^r.$$
	Thus, the minimal polynomial $P_2(t)$ of $\theta_2$ is 
	$$\Psi_p(t+2)=\sum_{r=0}^{(p-1)/2}\binom{p}{2r+1}\left(\frac{t}{4}+1\right)^{\frac{p-1}{2}-r}\left(\frac{t}{4}\right)^r.$$
	This implies 
	$$a_p(2)=\frac{p(p^2-1)}{24}.$$
	By (\ref{Eq. result of (i) of Thm.2}) we finally obtain 
	$$\det B_p(2)=\frac{(-1)^{\frac{p+1}{2}+\lfloor\frac{p-3}{4}\rfloor}}{24}\left(\frac{p-1}{2}\right)^{\frac{p-5}{2}}(p^2-1)p^{-(p-5)/2}.$$
	
	In view of the above, we have completed the proof of Theorem \ref{Thm. global properties of Bp(k)}.\qed 
	
	\medskip
	{\noindent\bf Declaration of competing interest} The authors declare that they have no conflict of interest.
	\medskip
	
	{\noindent\bf Data availability} No data was used for the research described in the article.
	\medskip

	\Ack\ The authors would like to thank two referees for their helpful comments. The authors also thank Professor Hao Pan for his encouragement.  
	
	We dedicate this paper to our advisor Professor Zhi-Wei Sun on the occasion of his 60th birthday.

	This research was supported by the Natural Science Foundation of China (Grant Nos. 12101321 and 12201291). The first author was also supported by Natural Science Foundation of Nanjing University of Posts and Telecommunications (Grant No. NY224107).

\end{document}